\documentclass[12pt]{amsart}
\usepackage{latexsym,fancyhdr,amssymb,color,amsmath,amsthm,graphicx,listings,comment}
\usepackage[section]{placeins}
\newtheorem{thm}{Theorem}
\newtheorem{lemma}{Lemma}
\newtheorem{coro}{Corollary}
\let\paragraph\subsection

\title{Mandelbulb, Mandelbrot, Mandelring and Hopfbrot }
\author{Oliver Knill}
\date{May 28, 2023, updated Jun 21, 2023}
\address{Department of Mathematics \\ Harvard University \\ Cambridge, MA, 02138 }
\subjclass{}

\keywords{Mandelbrot, Mandelbulb, Hopfbrot, Mandelstuff}

\begin{document}
\maketitle

\begin{abstract}
A topological ring $R$, an escape set $B \subset R$ and a family of maps $z^d+c$ defines the
degree $d$ Mandelstuff as the set of parameters for which the closure of the orbit of $0$ does
not intersect $B$. If $B$ is the complement of a ball of radius $2$ in $\mathbb{C}$ or $\mathbb{H}$,
it is the complex or quaternionic Mandelbrot set; in a vector space with polar decomposition 
$x=|x| U(\phi)$ like $R=\mathbb{R}^m$, the map  $z^d+c$ is defined as the map
$z=|z| U(\phi) \to z^d=|z|^d U(d\phi)$. We review some Jacobi Mandelstuff of quadratic maps
on almost periodic Jacobi matrices which have the spectrum on Julia sets. 
In a Banach algebra $R$, one can define $A^d=|A|^d U^d$ with $A=|A| U$ written as
the product of a self-adjoint $|A|$ and unitary element $U$.
In $\mathbb{R}^4$, the Hopf parametrization leads to the Hopfbrot, 
which has White-Nylander Mandelbulbs in $R=\mathbb{R}^3$ 
as traces and the standard Mandelbrot sets in $\mathbb{C}$ as co-dimension 
$2$ traces.  It is an open problem of White whether Mandelbulbs 
in higher dimensions are connected. 
The document contains an appendix with a proof of the Douady-Hubbard theorem.
\end{abstract}

\section{Euclidean Mandelstuff}

\begin{figure}[!htpb]
\scalebox{0.1}{\includegraphics{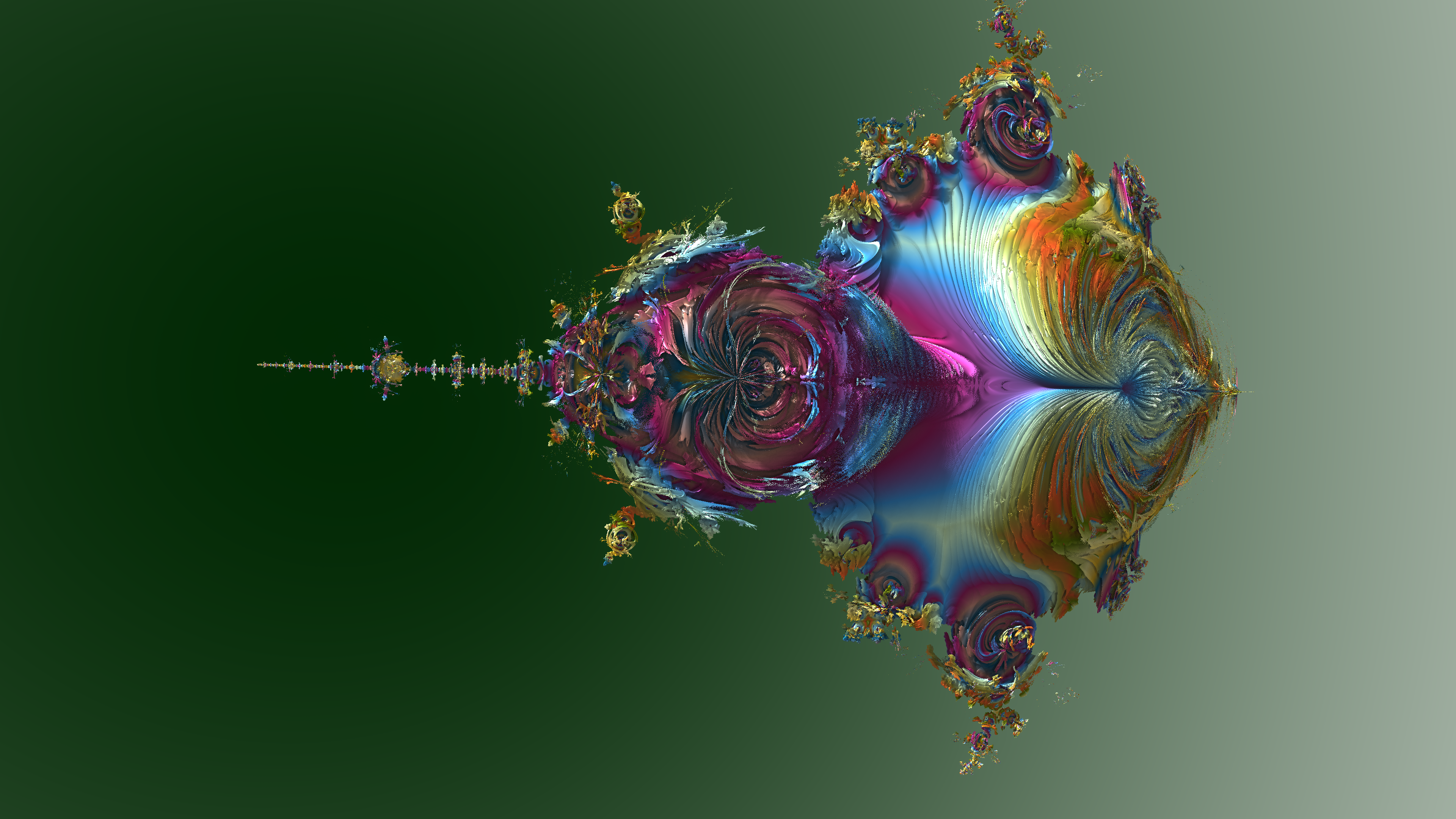}}
\caption{
The {\bf Mandelbug} is a degree 2 Mandelbulb in $\mathbb{R}^3$ generated with the 
``Mandelbulber" software. The set which has first been defined by Daniel White,
it has as the $z=0$ slice the usual Mandelbrot set in $\mathbb{C} \sim \mathbb{R}^2$. 
The object is therefore at least as complicated than the Mandelbrot set. 
While it is known that the Mandelbrot set is connected, it is not known 
whether the Mandelbug is connected.
}
\end{figure}

\paragraph{}
If $X=(r,\theta_1,\dots,\theta_{m-1})$ is a choice of {\bf spherical coordinates} in 
$R=\mathbb{R}^m$ using a {\bf sphere parametrization} $U(\theta_1, \dots, \theta_{m-1})$, 
and $X^d=(r^d, d \theta_1, \dots ,d \theta_{m-1})$, we have for every $c \in \mathbb{R}^m$ and $d \geq 2$, 
a map $T_c(X) = X^d+c$ of $R$ into itself. 
The points $c$ in $\mathbb{R}^d$ for which $T_c^n(0)$ stay bounded, 
is the {\bf Mandelbulb} of degree $d$ in that coordinate system. It is custom to use exponential notation 
$z^n$, even so it is not necessarily the power in the ring. 
The property $(z^d)^k=z^{dk}$ for example fails.
A specific case first used by Nylander-White is the parametrization
$\rho(\cos(\phi) \cos(\theta),\cos(\phi) \sin(\theta),\sin(\phi))$ which 
has the advantage that it produces on $\{ \phi=0\}$ the {\bf usual degree $d$ Mandelbrot set} in $\mathbb{C}$.
Motivated by its shape, one one could call it the {\bf White-Nylander Mandelbulb} 
the {\bf Mandelbug}. 

\paragraph{}
A different bulb in $\mathbb{R}^3$ is obtained with the spherical parametrization 
$$  \rho (\sin(\phi) \cos(\theta),\sin(\phi) \sin(\theta),\cos(\phi)) \; . $$
Any parametrisations of the sphere in $\mathbb{R}^n$ produces so a Mandelbulb in $\mathbb{R}^n$.
In four dimensions $m=4$, with the {\bf Hopf parametrization} of the $3$-sphere, 
we see the Hopf fibration \cite{Banchoff1990} and
get the {\bf Hopfbrot}, a name coined by Paul Nylander. By the way, a Povray 
implementation of {\bf Banchoff's visualization} of the 3-sphere had also been authored
by Nylander. 

\begin{figure}[!htpb]
\scalebox{0.80}{\includegraphics{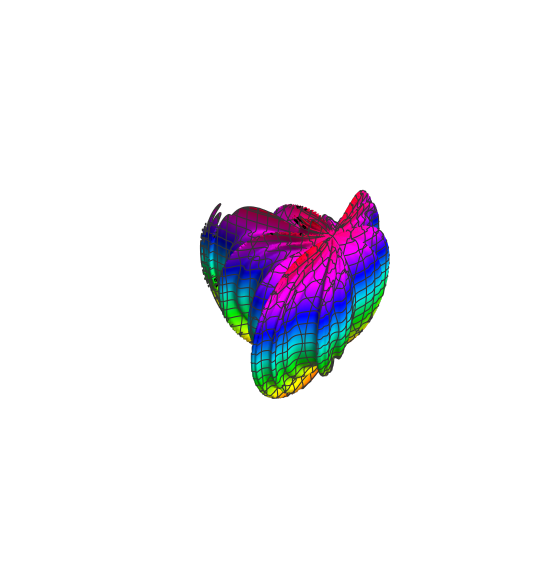}} 
\caption{
We see a picture showing a three dimensional slice of the four dimensional Hopfbrot. 
We have tweeted the 4 line Mathematica code for this picture on September 21, 2022. 
A variant is reproduced below. 
}
\end{figure}

\paragraph{}
The Hopf parametrization of the 3-sphere $x^2+y^2+z^2+w^2=1$ in $\mathbb{R}^4$ given by
$U(\theta_1,\theta_2,\phi)=(\cos(\phi) \cos(\theta_1),
                            \cos(\phi) \sin(\theta_1),
                            \sin(\phi) \cos(\theta_2),
                            \sin(\phi) \sin(\theta_2))$
with $\theta_1,\theta_2 \in [0,2\pi)$ and $\phi \in [0,\pi/2]$, 
we get for $\theta_1=0$ or $\theta_2=0$, two different White-Nylander Mandelbulbs. 
For $\phi=0$, it reduces on the co-dimension-2 slice $\{ z=0,w=0 \}$ to the usual 
Mandelbrot set. To display an approximation of the Hopfbrot, one can define the 
{\bf Green type functions} $G_n(c) = |T_c^n(0)|^2$ and looks at the regions 
$\{ G_n(c) \leq 4\}$. The following Mathematica code also can serve
as pseudo code, exactly telling how the set is computed. 

\lstset{language=Mathematica} \lstset{frameround=fttt}
\begin{lstlisting}[frame=single]
(* Hopfbrot solid   |T_c^5(0)|^2=4 cut at the plane   z=0*)
G[X_,m_,n_]:=Module[{Y=X}, T:=Block[{r=Sqrt[Y.Y],
  u=Arg[Y[[1]]+I*Y[[2]]],v=Arg[Y[[3]]+I*Y[[4]]],
  t=Arg[Y[[1]]/Cos[u]+I*Y[[3]]/Cos[v]]},
  Y=r^m{Cos[m*u]*Cos[m*t],Sin[m*u]*Cos[m*t],
        Cos[m*v]*Sin[m*t],Sin[m*v]*Sin[m*t]}+X]; 
  Do[T,{n}];Y.Y];
R=1.5; H[z_]:=RegionPlot3D[G[{x,y,z,w},2,5]<=4,
  {x,-R-1.5,R+0.1},{y,-R,R+.1},{w,-R,R+0.1},Boxed->False,
  Axes->False,PlotPoints->50, ColorFunction->Hue];H[0]
\end{lstlisting}

\paragraph{}
We can observe that the degree $d$ Hopfbrot in $4$-dimensions contains two 
co-dimension-1 Mandelbulbs of degree $d$, where one of them is the Mandelbug.
There is then also a co-dimension-2 Mandelbrot set.
Due to the fact that the $3$-sphere $S^3$ has a group structure $SU(2)$, 
there can be hope that it has nicer properties due to the enhanced 
algebraic structure. Nobody seems have looked however at the potential theoretical
aspect as carried out by Douady-Hubbard (an account of which has been added below 
as an appendix).

\paragraph{}
Since the standard $3$-sphere in $\mathbb{R}^4 \sim \mathbb{C}^2$ can be described 
as $|w|^2+|z|^2=1$, other {\bf power functions} like 
$(w,z)^2 \to (w^2,z^2)$ suggest themselves in $\mathbb{C}^2 \sim \mathbb{R}^4$.
That would not be an interesting choice however here because its Mandelstuff would 
just be the Cartesian product $M \times M$ in $\mathbb{H} \sim \mathbb{C}^2$. 
The connectivity question would then be obvious because in 
general, the product of two connected topological spaces is connected. 
In the usual quaternion case $R=\mathbb{H}$, the power is defined using quaternion 
multiplication. 

\section{Mandelstuff}

\paragraph{}
If $R$ is a topological ring and $B$ a closed {\bf escape set} in $R$, 
we can define the family of {\bf polynomial maps} $T_c(z)=z^d+c$. Each has an 
{\bf orbit} $O_c(z)=\{ T_c^k(z), k \geq 0 \}$, the {\bf forward orbit} of $z$ under
the dynamics. This defines the degree-$d$ {\bf ring Mandelstuff} 
$M=\{ c, \overline{O_c(0)} \cap B = \emptyset \}$. 
In a finite set $R$, we do not need a topology. In general, we might have the situation 
that the escape set is only approached in the limit. This is the case for 
$\mathbb{C}$ for example where $B=\{\infty\}$ is the {\bf point at infinity} of the 
{\bf Riemann sphere}. In a {\bf normed division algebra}, one can take $B=\{ |z|>2 \}$ 
so that $M$ is the set of parameters for which $0$ stays bounded. 
If $R=\mathbb{C}$, this is the Mandelbrot set. If $R=\mathbb{H}$ are the quaternions,
we get {\bf quaternion Mandelbrot set}. The later objects are implemented 
in the core language of the raytracer {\bf Povray}
very effectively. The following 2 lines produce a movie zooming into the Mandelbrot set.

\begin{tiny}
\lstset{language=Mathematica} \lstset{frameround=fttt}
\lstset{} \begin{lstlisting}[frame=single]
#declare c=clock; 
camera{location<-0.55,0.55,(1-c)/(1/5+c)> 
  look_at <-0.55,0.55,-5> 
  right <0,16/9,0> 
  up <0,0,1>}
plane{z,0 pigment{
   mandel 200 color_map{[0 rgb 0]
                        [(1-c)/6 rgb <1,c,1-c>]
                        [(1-c)/(3+3*c) rgb <1,c,0>]
                        [1 rgb 0]}} 
  finish{ambient 1}}
\end{lstlisting}
\end{tiny}

\paragraph{}
The {\bf quaternionic Mandelbrot set} $M_{\mathbb{H}}$ in which $|z|$ is 
the usual norm for a {\bf quaternion} $z$ is a product of the usual Mandelbrot set 
with a $2$-sphere. 
Note that if use in the ring $\mathbb{H}$, the 
decomposition $A=|A| U(A)$, where $|A|$ is the norm and $U(A)$ a unit quaternion, then 
$A^d = |A|^d U(A)^d$ is not the same than multiplying each of the Hopf angles by $d$ 
looked at before. In other words, this {\bf ring Mandelstuff} has little in common
with the Euclidean Hopfbrot defined before.

\paragraph{}
Ring Mandelstuff can also be studied in discrete and especially in finite situations. 
The analysis has then a more number theoretical
nature. For a {\bf finite ring} $R$ with multiplicative unit $1$, we could look at 
$B=\{1\}$ and look at $M(d,R)=\{ c, $, the orbit of $z \to z^d + c$ starting at 
$0$ does not contain $b$ $\}$. This is already interesting in a finite ring 
like $\mathbb{Z}_n$ or more geometrically in 
the $m$-dimensional vector space like the {\bf Galois field} $Z_p^m = Z_{p^m}$.
Polynomial maps $z \to z^g+c$ play an important role in {\bf cryptology} 
because they behave on $Z_p$ like {\bf pseudo random number generators}. 
The point $c=0$ is always in the Mandelbrot set. 
\footnote{The choice $0$ for the initial object and $1$ for the terminal object
suggests itself also when looking at categories which are also rings and playing a word game.}

\paragraph{}
Polynomial maps appear in cryptology because they allow to exploit the {\bf Birthday 
paradox phenomenon} to find collisions in an orbit to get factors. 
The {\bf Pollard $\rho$ method} uses the quadratic map $z \to z^2 + c$  to factor integers. 
See for example \cite{Riesel}.
Because the map is essentially random on the ring, there are many collisions, 
leading to a periodic {\bf attractor}  
playing the role of {\bf Julia sets}. Given an escape set $B$ or 
escape point, we can then look at the set of parameters for which the attractor does
not intersect $B$ and call this the Mandelbrot set of that escape set. In a topological 
ring, a natural {\bf escape set} would be the {\bf terminal object} $1$ and take as
a {\bf starting point} the {\bf initial object}.

\paragraph{}
Let us look at a specific case of
a finite field $Z_{p^2} \sim Z_p^2$ which is a discrete $2$-
dimensional vector space over $Z_p=\mathbb{Z}/(p \mathbb{Z})$.
We can in particular look at the {\bf Frobenius case} $T_c(z) = z^p+c$
in which case $T_c$ is a bijection of the ring. 
This means that the orbits are all periodic without pre-periodic component. 
We take $B=\{1\}$ and start to iterate orbits at the point $z=0$. 
The parameter point $c=0$ in the field is 
always in the Mandelbrot set because $z=0$ is a fixed point. 
There are $p(p-1)$ cases for which the orbit of $0$ never reaches $1$. 

\begin{lemma} In the Frobenius case and $m=2$ and where the map is $T_c(x)=x^p+c$,
the Mandelbrot set in $Z_{p^2}$ with initial point $0$ and escape point $1$
has $p^2-p+1$ elements. It covers the majority of the ring. 
\end{lemma}

\begin{proof}
We look at the orbit modulo $p$. The orbit of $T_c(x)=x^p+c$ modulo $p$ is because
of {\bf Fermat's little theorem} $x^p=x \; {\rm mod} \; p$ the same than the 
orbit of the {\bf linear map} $x \to x+c$. There are $(p-1)$ conjugacy classes 
$c$ for which this reaches $1$: every $a \neq 0$ and every $c+ap$  
with $c=0,\dots,p-1$ we have a dynamics which does not reach $0$.
Then there is also $c=0$ which always in the Mandelbrot set. 
\end{proof}

\paragraph{}
Of course, instead of $b=1$, one could use any other point different from $0$. 
In the Euclidean case, one has a natural point $b=\infty$ 
(the point at infinity after a one-point compactification), 
which has the property that $T_c(b)=b$. This can not be done in a finite ring.

\paragraph{}
How does the Mandelbrot set look like in this Frobenius case? 
One of the simplest set-ups is to look at a large prime $p$ and $m=2$ which 
produces a two dimensional Mandelbrot set. So, let $R=Z_p^2$ and 
$T_c(x)=x^2+c$ and $M=\{ c \in R, \; 1 \notin O_c(0) \}$. 
A natural question is to ask for example whether
$M$ is connected in the nearest neighbor topology.

\paragraph{}
If $X=\{0,1\}^{\mathbb{N}}$ is the set of all $0-1$ sequences 
with only finitely many $1$ and $S$ is a finite state set, we can think of the elements in 
$R = X \times S$ as a pair $(x,y)$ where the first part entries of the
{\bf tape of a Turing machine} $c$ and $y$ is the state of the machine. Define $B$ as
the set $X \times \{h\}$ where $h$ is the {\bf halte state}.
Instead of taking polynomial maps indexed by $R$ we can look at the class $T_c$ of maps
which are given by Turing machines $c$. Applying $T_c(x)$ is doing a {\bf computation step},
changing the tape and moving to a new state.
The {\bf Mandelbrot set} of this computation set-up is the set of Turing machines $c$
which do not halt when starting with the empty tape. 
This discrete Mandelbrot set is not computable
as Turing has shown. When seen like this, there are finite but potentially infinite dynamical systems 
for which the Mandelbrot set is not computable. In this context, one should mention
the {\bf Blum-Shub-Smale model} of computation, which uses real numbers for 
computation \cite{BSS} and which deals with the standard Mandelbrot set. 
The BSS machines are more general than Turing machines. The standard 
Turing machines, described first 1939, are finite, the BBS machines 
can work with real numbers. 

\paragraph{}
If $R$ is a real or complex matrix algebra in which we have a {\bf polar decomposition} 
$Z=|Z| U$ with $|Z|={\rm tr}(Z^*Z)^{1/2}$ (defined by the functional calculus for
the self-adjoint $Z^* Z$) and where $U$ is orthogonal or unitary, one can look at 
can look at $Z^d = |Z|^d U^n$ and define $M(R)$ as the set of $c$ such that 
$T_c(Z) = Z^d+c$ has a bounded orbit at $0$. This can be applied already to a small
space like $R=M(2,R)$, where the Mandelbrot set now can be seen as a subset of 
$\mathbb{R}^4$. With the usual power $Z^d=Z \cdot Z \cdots Z \cdot Z$ using matrix 
products, the Mandelbrot contains a product of standard Mandelbrot sets, the reason
being that we can always assume that $c$ is in Jordan normal form and that the orbit
stays in the commutative Banach algebra generated by $c$. The nilpotent part does not 
matter so that for example all nilpotent $c$ are in the Mandelbrot set. 

\paragraph{}
Let us look at the dynamics with $Z^d = |Z|^d U^d$ if $Z=|Z| U$. 
For symmetric matrices in $M(m,\mathbb{R})$, we have
$T(A)=A^d+c$ which remains in the class of symmetric matrices. The matrix Mandelstuff set $M$
is just the set of self-adjoint matrices for which all eigenvalues are in $[-1,1]$. 

\paragraph{}
On a finite ring $R$, any escape set $B$ defines Mandelstuff 
$M = \{ c \in R, \bigcup_{n \geq 0} T_c^n(0) \cap B=\emptyset \}$. For a
locally finite ring, one can look at a situation where the complement of $B$
is finite and have so a finite set $M$. To simulate a continuum like the Mandelbrot
set, we could look at a large number $N$, and take $R=(\mathbb{Z} + i \mathbb{Z})/N$ 
and define the power operation $z^d = [z^d]_R$ where $[x+i y]_R$ has the coordinates
$[N x],[N,y]$ where $[q]$ is the floor function. Taking $B=\{ |z|>2 \}$ gives for 
every $N$ a finite discrete Mandelbrot set. 

\section{Jacobi Mandelstuff}

\paragraph{}
In \cite{Kni93diss}, we looked at Jacobi operators
$L= a \tau + (a \tau)^*$ on $L^2(X)$ with $\tau f=f(T)$ defined by a 
{\bf measure theoretical dynamical system} $(X,T,m)$, where $(X,m)$ is
a probability space and where $a$ is a function in in $L^{\infty}(X,m)$. 
The operators $L$ are in a {\bf von Neumann algebra}, the 
{\bf crossed product} of the commutative von Neumann algebra $L^{\infty}(X)$ 
with the action. It is a type $II_1$ factor.  One can see $L$
as a bounded self-adjoint operator on the Hilbert space $L^2(X,m)$. Instead of looking 
at the operator $L$ on $L^2(X)$, one can also study for almost all $x \in X$ the 
operator $L_x u_n = a(T^n x) u_{n+1} + a(T^{n-1} x) u_{n-1}$ on $l^2(\mathbb{Z})$. 
If $T$ is ergodic, then the spectrum and spectral type is almost everywhere the same.
This is a frame work of {\bf random operators} in the sense of random variables taking 
values in a non-commutative Banach space \cite{Carmona,Cycon}. 
\footnote{The name ``random" is understood similarly as in probability theory. A 
random variable is a real-valued measurable function on a probability space. A 
random operator is an operator-valued measurable function on a probability space.}

\paragraph{}
If $(Y,S,n) = \phi(X,T,m)$ is a $2:1$ {\bf integral extension} of the dynamical system (meaning
$Y$ is a $2:1$ cover of $X$ and $S^2=T$ \cite{CFS,Parry80,Friedman}),
then $\phi^n$ converges in the metric $d(T,S)=m(\{x, T(x) \neq S(x) \})$ space 
of measure-preserving transformations to a unique fixed point, which is the 
{\bf adding machine} on the {\bf dyadic group of integers}. This dynamical system
is also known as the {\bf von Neumann Kakutani system} \cite{Kni93b}. 
It is an example of an ergodic dynamical system that is completely understood: it is an 
ergodic group translation on a commutative topological group and so has 
explicit discrete Koopman spectrum. 

\paragraph{}
The dyadic group of integers $\mathbb{Z}_2$ is one of the most important 
compact topological groups. Its embedding in the {\bf field of dyadic numbers} 
can be seen as the analog of the embedding of the usual integers $\mathbb{Z}$ in 
$\mathbb{R}$. The {\bf Pontryagin duality} between $\mathbb{Z}$ and 
$\mathbb{T}=\mathbb{R}/\mathbb{Z}$ has as an analog the 
Pontryagin duality between the dyadic group of integers $\mathbb{Z_2}$ and the 
{\bf Pr\"ufer group} $\mathbb{P}_2$, which are the rational dyadic numbers 
$k/2^n$ modulo 1, meaning the rational dyadic numbers in the circle 
$\mathbb{T}= \mathbb{R}/\mathbb{Z}$. Like $\mathbb{R}$, the field of dyadic numbers
$\mathbb{Q}_p$ is self-dual by a theorem of Tate. 

\paragraph{}
While in our {\bf real world}, the quotient 
$\mathbb{T}= \mathbb{R}/\mathbb{Z}$ is compact and the integers are discrete,
in the {\bf dyadic world}, the integers $\mathbb{Z}_2$ are compact and the 
quotient $\mathbb{P}_2$ is discrete. 
The dyadic analog of the {\bf addition} addition $x \to x+\alpha$ on $\mathbb{T}$ 
is a translation$n \to n+a$ on the group of dyadic integers. The smallest addition 
is the  {\bf adding machine} on $\mathbb{Z}_2$ which now is a measure preserving dynamical system. 
There is a minimal translation unlike in the traditional world, where the compact space $\mathbb{T}$ 
has no smallest ergodic translations $T(x)=x+\alpha$ on $\mathbb{T}$. 
One knows everything about this almost periodic system $(X,T,m)$.
The Koopman spectrum of the unitary $\tau$ is the dual group of $X$, the 
{\bf Pr\"ufer group} $\mathbb{P}_2$. It is the unique dynamical system which has a 
square root that is isomorphic. It is nice to have a compact set of integers and so 
a compact topological group in which there is a smallest translation. 
Having a {\bf smallest translation}, space is quantized unlike in the real world, 
where we can produce arbitrary small translations $x \to x+\alpha$. 

\paragraph{}
The renormalization on the base space of dynamical systems can be lifted to the fiber
bundle of operators \cite{Kni95}. \footnote{The picture is that over every point $T$,
there is a fiber given as a Von Neumann Algebra defined by $T$.} If we start with a Jacobi operator
$L$, applying the map $L^2+c Id$ produces operators which are no more Jacobi operators in general.
We therefore apply the inverse operation, which has two branches. This works nicely and produces again 
Jacobi operators.  Given $c$ outside the Mandelbrot set, there are two operators
$D=m^{\pm} \sigma + (m^{\pm} \sigma)^*$ such that $D^2+c=L$. The new functions $m^{\pm}$ 
are in the real case the {\bf Titchmarsh-Weyl functions} related to the stable and unstable 
{\bf Oseledec spaces} of the hyperbolic $SL(2,\mathbb{C})$-transfer cocycle obtained 
from $L$. We write $\phi(L)=D$. We have now a renormalization map $(T,L) \to (S,D)$,
where $S^2=T$ and $L=D^2+c$ (more precisely two copies of $L$) on each ergodic 
component one. Now define the {\bf Jacobi Mandelbrot set} as the complement of 
complex energies $c$ for which the operators $\phi^n(L)$ converge to a bounded operator. 

\paragraph{}
On the spectral level, the spectrum as a set satisfies $\sigma(L) = \sigma(D)^2 + c$. 
This means that we necessarily need the Jacobi Mandelbrot set to include the standard 
Mandelbrot set. We know that for $|c|$ large enough, the map $\phi$ is a contraction 
in the space of bounded operators. 
Let $JM$ denote the Jacobi Mandelbrot set. So far we only know

\begin{thm} The set $JM$ satisfies $M \subset JM \subset \mathbb{C}$ and 
is compact.  \end{thm}

\paragraph{}
We believe that $JM=M$, but we did not show that yet. The difficulty we battled when working on this 
as a graduate student was to control the convergence of the two contractions obtained by 
inverting $T_c(L)=L^2+c$. When working in $\mathbb{C}$, we deal with an {\bf iterated function system}
in an operator space: there are two contractions and the Julia sets for $c$ outside the Mandelbrot set are
known to be {\bf Cantor dust} of zero measure. In the quantum version, where the iteration 
is made in fiber bundle over the space of dynamical systems, where each figure is 
$L^{\infty}(X)$ encoding the operator $L$, we also have an iterated function system for 
large $|c|$ and so convergence, we do not know that in general. We can not just take a
vector $\psi$ and look at $(\psi, \phi^n(L) \psi) \in \mathbb{C}$ because the renormalization
map is chaotic on the attractor $J_c$ in the Banach space, the analog of the Julia set
in the complex case.  But it is very likely that we can for every complex $c$ outside the 
Mandelbrot set just find a suitable metric on the von Neumann algebra so that the two branches
of the iterated function system are contractions and establish so that there is a Cantor set of 
random Jacobi operators which are invariant under the map $T_c(L)=L^2+c$.

\paragraph{}
We should add that in a non-ergodic set-up while working with bounded 
operators on $l^2(\mathbb{N}$, the almost periodic matrices related to this 
{\bf non-commutative Mandelbrot story} had been studied in 
\cite{BGH,BBM}. If one looks at these operators one gets half infinite potentials
which have the same hull than the sequences we were looking at. What was new in 
our approach is that we worked within the class of Jacobi operators defined over
dynamical systems and could lift the quadratic map to an attractor of operators 
which have the property that $L^2$ restricted to half of the space is again of the 
same form. In some sense, this is a non-commutative version of the quadratic map. 
Instead of working in the commutative algebra $\mathbb{C}$, we had worked directly 
in an infinite dimensional non-commutative von Neumann algebra. We therefore almost
for free got the nature of the nature of the hull of the almost periodic sequences
generated by \cite{BGH,BBM}. We once thought being able to prove that {\bf all}
operator $L_x$ in this model have purely singular continuous spectrum. One only has 
to exclude point spectrum. By Kotani theory and because the Lyapunov exponent is the 
potential theoretical Green function on the Julia set (which is zero exactly on the 
Julia set) and the Julia set for $c$ outside the Mandelbrot set is a Cantor set, 
the spectrum of $L_x$ never contains any absolutely continuous part. Also, unlike
in random operator theory, the spectrum of $L_x$ should be the same for {\bf all x}
in the dyadic group of integers and not only for {\bf almost all}.

\paragraph{}
The work on Jacobi operators over the group $\mathbb{Z}_2$ of dyadic integers is 
far from finished. We saw for example that we can interpolate 
B\"acklund transformations with isospectral deformations \cite{Kni93a} so that we are
able to connect $L$ with $L(T)$ within the isospectral set. This suggests
that the {\bf isospectral set} of Jacobi operators is topologically the 
topological group of the {\bf dyadic soleonid} which is the dual of the 
{\bf dyadic rationals} $\mathbb{Q}_2$. In the periodic case, the isospectral set is a finite
dimensional torus and an explicit algebro-geometric picture explaining how the auxiliary spectrum
(divisors on a hyperelliptic curve) serve as coordinates and that the Abel-Jacobi map translates
between this picture and the torus. This picture integrates the isospectral nonlinear Toda flows 
$L'=[B(L),L]$.  The Fourier transform then allows to look
instead at operators $\hat{L} u(x) = \hat{a}(x) u(x+1)+\hat{a}(x-1) u(x-1)$ on 
$l^2(\mathbb{Q}_2)$ which have the property that $\hat{L}^2 + c$ produces the operator 
$\hat{L}' u(x) = \hat{a}(x) u(x+2) + \hat{a}(x-2) u(x-2)$. 
We have now the strange situation that the Fourier transform of the operator on 
$\mathbb{Z}_2$ produces again operators of the same type. The function 
$a \in C(\mathbb{Z}_2)$ produces so a new function $\hat{a}$ on $\mathbb{Q}_2$
which produces sequences which have hulls which are again $C(\mathbb{Z}_2)$. 
There is a lot of symmetry in this non-commutative version of the quadratic map. There is much 
likely more to it as all happens in the field of dyadic numbers which like $\mathbb{R}$ is
self-dual. It is an interesting question whether the isospectral set of operators on the solenoid
then can be extended to the dyadic field. An other question is whether
{\bf integrable system} which can be written as a Lax pair 
(and almost all systems which are integrable do), can be embedded as an integrable 
system on the solenoid and so the dyadic field, providing so a universal vessel for integrability. 
\footnote{Tate's theorem states that all p-adic fields $\mathbb{Q}_p$ produce self-dual 
topological groups under addition similarly as $\mathbb{R}=\mathbb{Q}_0$ is a self-dual topological
group.}

\section{Historical addenda}

\paragraph{}
The history of Mandelbrot starts with the mathematics of {\bf complex dynamics}
\cite{Alexander1994,Carlson,Mandelbrot2004,MilnorNotes,Schleicher2009,Audin2014}.
About 100 years ago, Ernst Schr\"oder, Gaston Julia and 
Pierre Fatou pioneered the study of iterating maps in the complex maps. 
John Hubbard writes in \cite{TanLei2000}
{\it Holomorphic dynamics is a subject with an ancient history: 
Fatou, Julia, Schroeder, Koenigs, Bottcher, Lattes, 
which then went into hibernation for about 60 years, and came back to 
explosive life in the 1980's.}.

\paragraph{}
The question of connectivity of the Mandelbrot set goes back to Mandelbrot who 
first thought that the object is not connected (1985). Also Milnor conjectured
independently connectivity and that the Hausdorff dimension of the Mandelbrot 
set to be $2$ \cite{Milnor1989}. Douady and Hubbard proved the connectivity 
in 1981-1982. The Orsay notes (published \cite{DouadyHubbard} 
of A. Douady and J.H. Hubbard state that the notes were
based on lecture notes of Douady for a course "Holomorphic Dynamical systems 1983-1984).

\paragraph{}
Daniel White who defined the three dimensional version degree 2 Mandelbug
first, asked already about the connectivity of the object. This is still unexplored.
Paul Nylander wrote me about the pioneer days: 
"The earliest pioneer was Daniel White, and Thomas Ludwig was the first to create 3D renderings. 
I may have been the third person to get involved, as I began emailing Daniel and Thomas and 
sending my renderings to them, but I did not post anything to the discussion thread until 
much later. Other early contributors who were notable to me were David Makin and 
Krzysztof Marczak.". Marczak is credited in \cite{Odyssey} as an author of the Mandelbulb
software. 

\paragraph{}
It appears that various mathematicians have looked at Mandelbrot before. 
Brooks and Matelski who worked in Kleinian groups 
have given a talk at Harvard, while Mandelbrot was
at the department. They have produced computer generated pictures already and
showed them around. Also John Hubbard appears to have experimented with the object 
before. In the article \cite{Horgan2009}, one can see that even earlier,
in the 1950ies, Riesz might already have experimented with objects close to the 
Mandelbrot set. 
As for the Mandelbulb, it seems that after some experimentation of Rudy Rucker
and Jules Ruis (a document of Jules on ``Fractal Trigeometry" from March 14th 2012 states that
he created a Mandelbulb on a BBM-15 on December 29, 1997), 
it was Daniel White who first looked at the degree 2 Mandelbulb. 
He also named the object. John Nylander looked first at higher degree cases and 
generated rational expressions. Nylander also looked at the Hopfbrot. 

\paragraph{}
Why have Mandelbulbs not been investigated more mathematically? One reason could be that it 
is too hard. One would have to find an extension of the B\"ottcher-Fatou Lemma 
to the real case, where we have no compactness. 
We would have to show the Julia sets are non-empty and that the 
Green functions exist. Various topological statements need to be reworked, generalized or 
modified. We look in an appendix at a detailed version of the proof of Douady and Hubbard about
the connectivity of the standard Mandelbrot set. 

\paragraph{}
An other reason for the lack of interest could be that the
roots of the Mandelbulb are not in the hands of pure mathematicians.  
Rudy Rucker, a science fiction writer and hacker (in the sense of a gifted and creative programmer)
first experimented with it, Jules Ruis, an independent researcher in 1997 looked 
at first formula 
Daniel White, a math enthusiast with music and computing degree developed it further
in 2007; then in 2009, Paul Nylander, a mechanical engineer polished the now established version.
It would be interesting to see what would have happened if a famous mathematician 
would have conjectured that the Mandelbulb is connected. Without doubt, 
there might be more papers about it. Mathematics is also a social activity and 
social creatures are driven by social values like names who are behind conjectures.
Independent of fashion, it should be clear however that the connectivity problem is 
interesting. The object has not only drawn attention on Youtube. The book cover
\cite{Odyssey} features the Mandelbulb. 


\paragraph{}
As for the area of the Mandelbrot set, the question of estimating the area has been 
stated by many, in particular hobby mathematicians.
Since $M$ must be contained in  $[-2,1] \times [-3/2,3/2]$ which has area $9$, one
has a simple finite upper bound, a bit better already than the area $|\{|z| \leq 2\}|=4\pi$.
Better estimates can be obtained by looking at polynomials and compute the areas
using computer algebra systems. Here is how to compute these polynomials

\lstset{language=Mathematica} \lstset{frameround=fttt}
\lstset{} \begin{lstlisting}[frame=single]
p[0_,z_]:= z; p[n_,z_]:=p[n-1,z]^2+z;
ContourPlot[Abs[p[8,x+I y]]==2,{x,-2,1}, {y,-3/2,3/2}]
f=p[2,x+I*y]*p[2,x-I*y] // Expand
\end{lstlisting}

The polynomial $p_2(x,y)$ for  example is
\begin{tiny}
$f=x^8+4 x^7+4 x^6 y^2+6 x^6+12 x^5 y^2+6 x^5+6 x^4 y^4+14 x^4 y^2+5 x^4+$
    $12x^3 y^4+4 x^3 y^2+2 x^3+4 x^2 y^6+10 x^2 y^4+2 x^2 y^2+x^2+4 x y^6-2 x y^4+$
    $2 x y^2+y^8+2 y^6-3 y^4+y^2$.
\end{tiny}

A computer algebra system then can compute for example

\lstset{language=Mathematica} \lstset{frameround=fttt}
\lstset{} \begin{lstlisting}[frame=single]
Area[ImplicitRegion[f<4,{x,y}]]
\end{lstlisting}

which gives the result $4.14653$.
Using interval arithmetic, one can like this get computer assisted 
(rigorous) upper bounds.

\paragraph{}
An other method is to use {\bf Monte Carlo} which gives typically values around 1.5.
Monte Carlo integration is by the law of large numbers a numerical method for
computing the Lebesgue measure. Here is a Mathematica line which produces estimates using Monte-Carlo. 

\lstset{language=Mathematica} \lstset{frameround=fttt}
\lstset{} \begin{lstlisting}[frame=single]
M=Compile[{x,y},Module[{z=x+I y,k=0},
  While[Abs[z]<2.&&k<999,z=N[z^2+x+I y];++k];Floor[k/999]]];
n=10^6; 9*Sum[M[-2+3*Random[],-1.5+3*Random[]],{n}]/n
\end{lstlisting}

\paragraph{}
\cite{EwingSchober} give bounds $[1.3744,1.7274]$ where the left side is $7\pi/16$ and
the upper bound uses partial sums.  \cite{BCGN} gives the upper bound $1.68288$
One particular source gives the bounds $[1.50311,1.5613027]$.

\section*{Appendix A: From the Sphere to the Mandelbulb} 

\paragraph{}
In this 30 minutes talk, we build a visual bridge from Greek mathematics to 
the present, where 3D printing and 3D scanning have become affordable. 
We live in an exciting time, where several revolutions happen simultaneously.
This short document summarizes the main points of the talk. 
\footnote{These are notes from a talk from 2013. The text is unchanged except 
for fixing typos and where indicated. 
The historical assessment had been what I had in mind at that time.}

\paragraph{}
Science, technology and education experiences a time of rapid change. 
Things which would have been considered impossible 20 years ago are now possible or 
available on the level of consumer technology. Jeremy Rifkin once coined the term 
{\bf "third industrial revolution"} and called it a time, where manufacturing 
becomes digital, personal and sustainable. We are in the third of the three 
industrial revolutions which one could coin {\bf "steam, steal and star"} revolutions. 
Media revolutions like {\bf Gutenberg's press}, the {\bf telegraph} are now 
topped with the emergence of the {\bf internet} and social networks. 
We had {\bf print, talk and connect} revolutions. Information storage media sizes
have exploded. We have come a long way since Clay tablets, papyrus, paper, vinyl. 
These physical media have in rapid succession been replaced with CDs and DVDs, 
tapes, hard drives and now memory sticks, SSD's. 

\paragraph{}
The three {\bf optical, magnetic and electric} revolutions came within a few decades only.
In education, we have seen {\bf "New math"} in the sixties, the {\bf "Math wars"}
in the eighties and {\bf Massive Open Online Courses} (MOOC's) at the beginning of the 21st century. 
These revolutions happened for {\bf generation X, generation Y and generation Z}. For media, we have
seen the emergence of {\bf photograph}, {\bf film} and now witness {\bf 3D scanning and printing}: 
these are {\bf "picture, movie and object"} revolutions. We have first explored the 
{\bf macro scale} (cosmos) by measuring out planets and stars, then the 
{\bf micro-scale} (atomic level) with physics, chemistry and biology and conquer now the 
{\bf meso-scale}, the level of our daily lives, which also includes biology. 

\paragraph{}
What does this have to do with math? Describing complex geometric objects by 
mathematical language is now possible thanks to {\bf 3D scanners}. 
We can copy the objects using {\bf 3D printers}. 
Understanding the motion of stars led to {\bf calculus} and {\bf topology} gaining insight into 
quantum mechanics required {\bf analysis} and {\bf algebra}, the meso-scale feeds from 
knowledge about {\bf geometry} and {\bf computer science}. \\

\paragraph{}
This section uses work done in \cite{CFZ} and \cite{Slavkovsky}.
Even so we live in a revolutionary time, we should not forget about the past.
Early achievements in mathematics and technology made the present progress
possible. Archimedes, who would now be 2300 years old, was an extraordinary 
mathematician and engineer. We do not have a picture of him, but artists 
have tried to capture the legend. A common scene with Archimedes is 
the moment of his death, when Syracuse was sacked. Archimedes tomb stone was decorated
with a symbol of his most important achievement: the computation of 
the {\bf volume of the sphere}. 

\paragraph{}
Archimedes had realized that if one cuts a sphere at height $z$
then the area of the cross section $\pi (1-z^2)$ is the same than the difference
of the area of a cross section of a unit cylinder $\pi$ and the cross section with
a cone $\pi z^2$. The volume of the half sphere is therefore the difference
between the volume of the cylinder $\pi$ and the volume of the cone $\pi/3$ which 
gives $2\pi/3$. While we know now how to compute the volume of the sphere using calculus, 
this had not been known before Archimedes. 

\paragraph{}
In order to translate this discovery into modern times, 
let us look at a problem which we presently do not know how to solve. 
It might be easy with an Archimedean type idea, but it might also be impossible,
like the {\bf quadrature of the circle}. 
\footnote{We understood ``impossible" in a way that there is no explicit formula,
that the number either transcendental unrelated to any other transcendental number we 
know or to express in a fancy manner that the volume is not a so called ``period" as
explained below.}
The object we are going to look at is called the {\bf Mandelbulb set} $M_8$. 
Unlike the {\bf Mandelbrot set}, which was discovered 33 years ago, the
set  $M_8$ is an object in space has been discovered only recently.

\paragraph{}
Already {\bf Rudy Rucker} had experimented with a variant $M_2$ in 1988 in Mathematica.
{\bf Jules Ruis} wrote me to have written a Basic program in 1997 and also
3D printed the first models in 2010. 
The first mentioning of the formulas for $M_2$ used today is by {\bf Daniel White} mentioned in a 
2009 fractal forum. [Added in Sepember 2022: 
Paul Nylander was the first to look at higher degree bulbs $M_d$ like $M_8$ 
as well as higher dimensional versions like Hopfbrot in $\mathbb{R}^4$. Nylander also first created 
algebraic expressions for the iteration.]
A problem one could ask is: is there a closed expression for the volume of $M_k$? 
What is a closed expression? The problem to find an expression of
$\zeta(2)= 1+1/4+1/9+1/16+ \dots $ is called the {\bf Basel problem}. It was 
solved by Euler, who found that the limit is $\pi^2/6$. While one knows similar formulas for sums 
$\zeta(k) = 1+1/2^k + 1/3^k + \dots $ for all even $k$, one does not have a closed
formula for $k=3$ for example. Mathematicians call $\zeta(2)$ a {\bf period}, a number
which can be expressed as an integral of an algebraic function over an algebraic domain.

\paragraph{}
The You-Tube star Mandelbulb $M_8$ is defined as
follows: for every three vector $c$ define a map $T_c$ in space which maps a point
$x$ to a new point $T_c(x) = x^8 + c$, where $x^8$ is the point with spherical coordinates
$(r^8,8 \phi,8\theta)$, if $x$ had the spherical coordinates $(r,\phi,\theta)$. 
The set consists now of all parameter vectors $c$ for which iterating $T_c$ again and again 
starting at $0$ leads to a bounded orbit. For example, if we take $c=(0,0,0)$ and start
with $(0,0,0)$, then we remain at $(0,0,0)$. Therefore, $c=(0,0,0)$ is in $M_8$.
Now take $c=(0,0,2)$. Lets look at the orbit $(0,0,0)$, which is 
$(0,0,0),(0,0,2),(0,0,2^8),...$ so that $(0,0,2)$ is outside  of $M_8$. 

\paragraph{}
The Mandelbulb set is mathematically not well explored. It is most likely connected and simply 
connected like the Mandelbrot set but there could be surprises. Not much mathematical seems to
been proven. Math reviews and ArXiv show no references yet.
Here is an exercise for you: prove that the boundary of the Mandelbulb set is contained 
in the shell defined by the sphere of radius $2$ and $1/2$. Now try to improve
this bound. Low resolution picture of the Mandelbulb set can be computed with Mathematica's 
RegionPlot and spit out a STL file:
\footnote{In 2013, I had called it the White Mandelbulb set. But Paul Nylander defined first the higher genus 
versions like this one} 
\footnote{The Mandelbulb is available on thingiverse https://www.thingiverse.com/thing:26324}

\begin{figure}[!htpb]
\scalebox{0.22}{\includegraphics{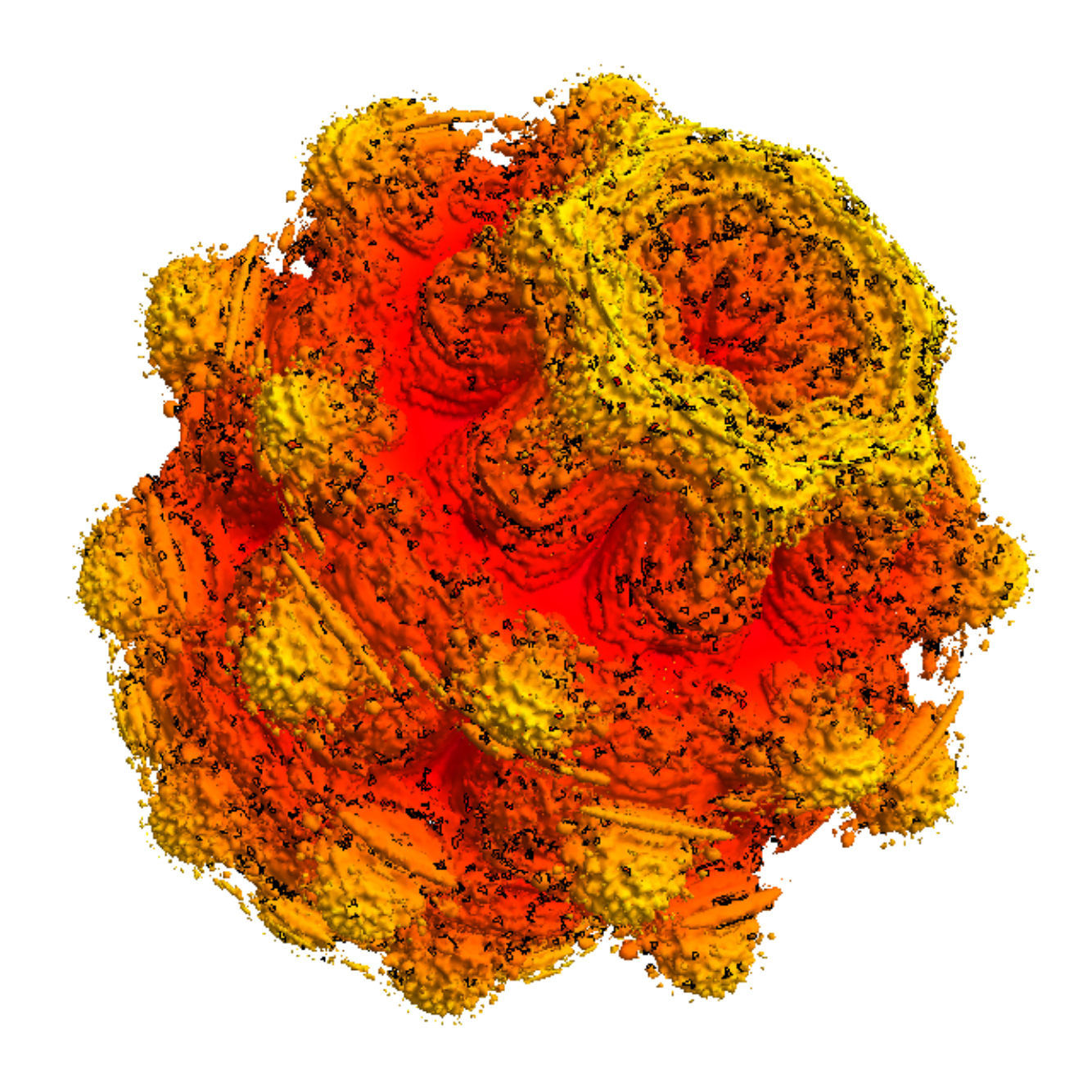}}
\scalebox{0.22}{\includegraphics{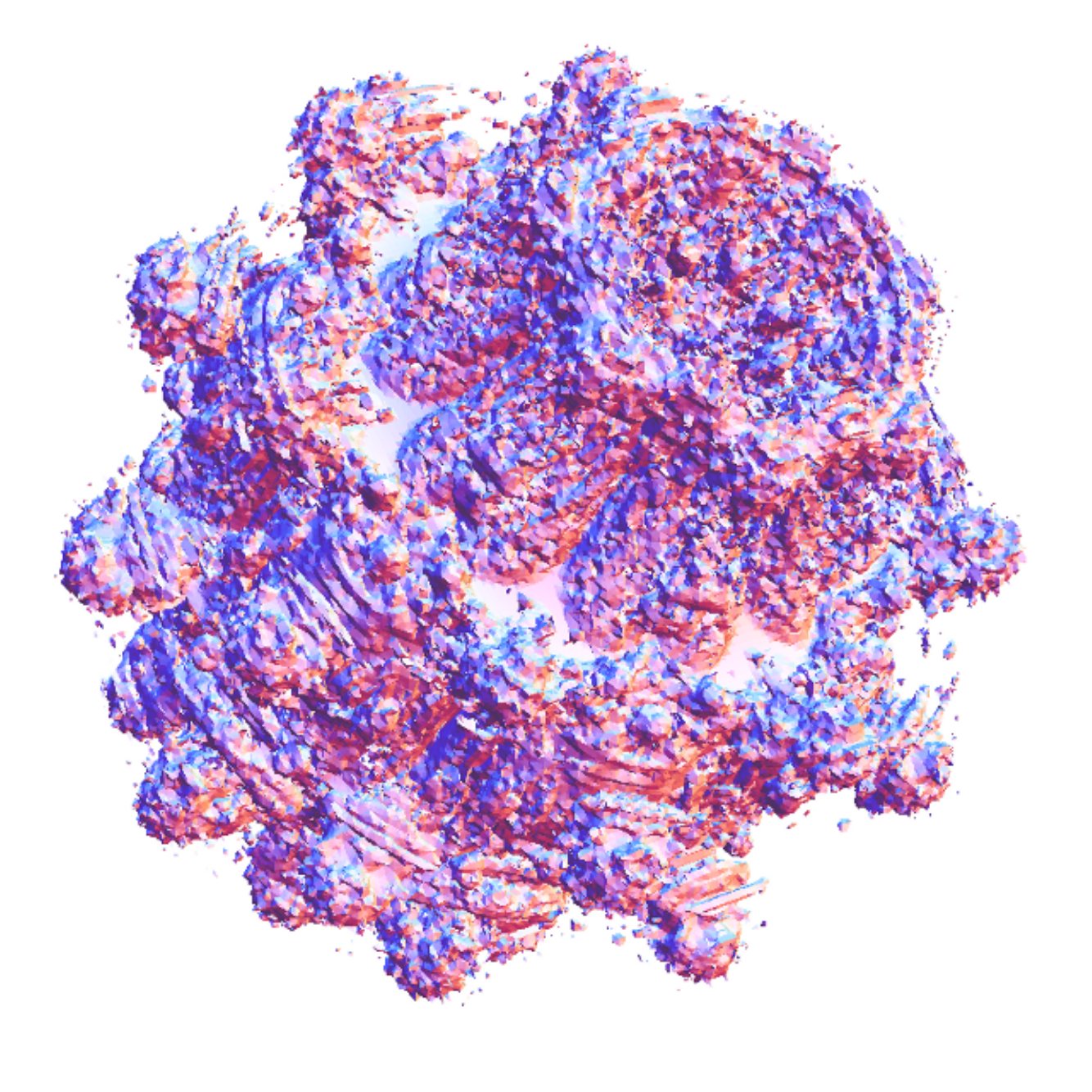}}
\caption{
The White-Nylander Mandelbulb set 
$M_8$ seen as a
printable STL file. White asked in 2008: Is it connected? Simply connected?
Locally connected? Is the complement simply connected? We asked here whether
the volume of $M_8$ is a period (a branch of mathematics called measure theory
assures that the volume of of the set exists). Printable Mandel and Julia bulbs
have been produced in \cite{mandelbulb1}.
}
\end{figure}

\paragraph{}
Some pictures, animations and movies shown in this part of the lecture were borrowed from
my talk given at a ICTP workshop in Trieste and the article \cite{CFZ}.
Most of the pictures were done with a 3D printer in mind. 
{\bf Thinking like a printer makes pictures nicer.} One can not print surfaces of zero thickness
for example. Graphics needs to be redesigned. This can pay off. The animations of the 
Calabi-Yau surfaces for example show the object in new beauty. An other side effect of 3D printing
is that we can now plot solid graphs of functions which can be placed in Google Earth
to fly around. We illustrate ``impossible figures" like the {\bf Escher stairs}, the 
{\bf Penrose triangle}, as well as the Mandelbrot set as an island in the Boston Harbor 
in this virtual reality setup. 

\paragraph{}
We first look at the problem of reconstructing a three
dimensional object from projections. This is called the {\bf structure
from motion problem}. We use the affordable Microsoft Kinect device as a scanner and 
illustrate this.  The program Artec Studio 9.1, reconstructs the object 
almost in real time, during the talk. \\

\paragraph{}
{\bf Acknowledgements:} Thanks to {\bf Gloria Li} for the invitation to
talk at the Science Initiative Competition (2013). My talk with title
``From Archimedes to 3D Printers" covered work done which was done
with {\bf Elizabeth Slavkovsky} \cite{KnillSlavkovskyArchimedes,CFZ}
(Elizabeth wrote a thesis \cite{Slavkovsky} which I advised)
and was inspired by a workshop on 3D printing in Trieste by 
{\bf Enrique Canessa}, {\bf Carlo Fonda} and {\bf Marco Zennaro},
where {\bf Daniel Pietrosemoli} showed how to use the Kinect for
3D scanning. Some of the slides are reused from the talk given at the ICTP 
in Trieste. Thanks to {\bf Anna Zevelyov} at Artec for an educational evaluation 
license of Artec Studio 9.1 which could be demonstrate live during the talk on a
Macbook air running Windows 7. Update: October 7, 2013: thanks to 
{\bf Jules Ruis} and {\bf Rudy Rucker} for comments on the history of the Mandelbulb. 
Update September 15, 2021: thanks to {\bf Paul Nylander} for answering more of my
questions about the origin of the bulb. 
{\bf Added September 2022:} The original document from 2013 is still available on my
website. In this version, some typos were fixed.

\section*{Appendix B: the Douady-Hubbard proof}

\paragraph{}
The proof of the theorem of Douady and Hubbard \cite{DouadyHubbard82} of the
connectivity of the Mandelbrot set needs some concepts from 
topology and complex analysis and topology. This section had been a handout
from a Math 118R course given at Harvard in 2005. 
I myself learned the Douady-Hubbard proof from {\bf Jochen Denzler}, who was an undergraduate
mathematics student like me at et ETH and who had presented it in an undergraduate
seminar run by J\"urgen Moser in 1995 at ETH Z\"urich. (I had presented in that 
seminar a stability theorem of Rabinowitz for multi-dimensional holomorphic 
complex dynamical systems which dealt with the question under which conditions
such a holomorphic dynamical system can be linearized near a fixed point. 

\begin{figure}[!htpb]
\scalebox{0.80}{\includegraphics{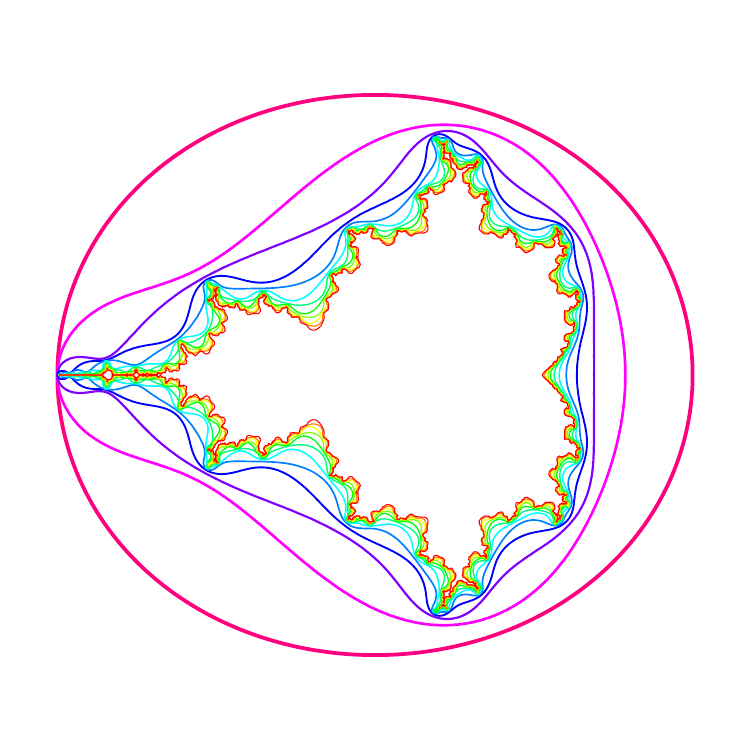}}
\caption{
We see some level curves of Green function. 
We can also illustrate it using level curves. Computing
such level curves and the area they enclosed can be used to give rigorous 
upper bounds on the {\bf area} of the Mandelbrot set. 
}
\end{figure}

\paragraph{}
A function $f: \mathbb{C} \to \mathbb{C}$ is called {\bf analytic} in an open set 
$U$ if the derivative $f'(z)=\lim_{w \to 0} (f(z+w)-f(z))/w$ of $f$ exist and are 
continuous at every point $z$ in $U$. This means that for 
$f(z)=f(x+iy)=u(x+iy)+iv(x+iy)$, the partial derivatives $u_x,u_y,v_x,v_y$
are all continuous real-valued functions on $U$. In that case, $u(x,y),v(x,y)$
are {\bf harmonic}: $u_{xx}+u_{yy}=0$. The open set
$B_r(z)= \{ w \; | \; |z-w| < r \; \}$ is a neighborhood
of $z$ called an {\bf open ball}. A sequence of analytic maps $f_n$ 
{\bf converges uniformly} to $f$ on a compact set $K \subset U$, if 
$f_n \to f$ in the space $C(K)$ of continuous functions on $K$. This 
means $\max_{x \in K} |f_n(x)-f(x)| \to 0$. 
A family of analytic maps ${\mathcal{F}}$ on $U$ is called {\bf normal}, if 
every sequence $f_n \in {\mathcal{F}}$ has a sub-sequence which {\bf converges uniformly}
on any compact subset of $U$. The limit function $f$ does not need to be in ${\mathcal{F}}$.
Normality is just pre-compactness: a set ${\mathcal{F}}$ is normal, if and only if 
its closure is compact. The {\bf theorem of Arzela-Ascoli} (see \cite{AhlforsComplexAnalysis} p. 224) 
states says that normality of ${\mathcal{F}}$ is equivalent to the requirement that each $f$ is
{\bf equi-continuous} on every compact set $K \subset U$ and if for every $z \in U$, the set 
$\{ f(z) \; |\; f \in {\mathcal{F}} \}$ is bounded. 
The {\bf Fatou set} of $f$ is defined to be the set of $z \in \mathbb{C}$ for which the family 
$\{f^n\}_{n \in N}$ is normal in some open neighborhood of $z$. 
The {\bf Julia set} is the complement of the Fatou set. A set is called {\bf locally compact}, 
if every point has a compact neighborhood. In the plane, a set is compact if and only if it is
bounded and closed. A subset is closed, if and only if its complement is
open. A subset $U$ is open, if for every point $x$ in $U$ there is a
ball $B_r(x)$ which still belongs to $U$. For notions of complex analysis, see 
\cite{AhlforsComplexAnalysis,Conway1978}.

\begin{lemma}[B\"ottcher-Fatou lemma]
Assume $f(z)=z^k+a_{k+1} z^{k+1} + \cdots$ with $k \geq 2$
is analytic near $0$. Define $\phi_n(z)=(f^n(z))^{1/k^n}=z+a_1 z^2 + \dots$.
In a neighborhood $U$ of $z=0$, define 
$\phi(z)=\lim_{n \to \infty} \phi_n(z): U \to B_r(0)$.
It satisfies $\phi \circ f \circ \phi^{-1}(z)=z^k$ and $\phi(0)=0$ and $\phi'(0)=1$.
\end{lemma} 

\begin{proof}. We show that $\phi_n$ converges uniformly. 
The properties $\phi(f(z))=\phi(z)^k$ as well as $\phi(0)=0$ and $\phi'(0)=1$
follow directly from the assumptions. The function
$$ h(z):=\log(\frac{f(z)^{1/k}}{z}) $$
with the chosen root $f(z))^{1/k}=z+O(z^2)$
is analytic in a neighborhood $U$ of $0$ and there exists a constant $C$
such that $|h(z)|\leq C|z|$ for $z \in U$. $U$ can be chosen so small
that $f(U) \subset U$ and $|f(z)| \leq |z|$.
We can write $\phi(z)$ as an infinite product
$$ \phi(z)= z \cdot \frac{\phi_1(z)}{z} \cdot \frac{\phi_2(z)}{\phi_1(z)} 
              \cdot \frac{\phi_3(z)}{\phi_2(z)} \dots  \; . $$
This product converges, because
$\sum_{n=0}^{\infty} \log \frac{\phi_{n+1}(z)}{\phi_n(z)}$ converges
absolutely and uniformly for $z \in U$:
$$|\log \frac{\phi_{n+1}(z)}{\phi_n(z)}| 
     = |\log \left[ \frac{ (f \circ f^n(z))^{1/k} }{f^n(z)} \right]^{1/k^n}| 
     = \frac{1}{k^n} \cdot | h(f^n(z)) | \leq \frac{1}{k^n} C \cdot |f^n(z)| 
     \leq \frac{C \cdot |z|}{k^n}  \; . $$
\end{proof}

\begin{coro}
If $c \mapsto f_c(z)$ is a family of analytic maps such that
$c \mapsto f_c(z)$ is analytic for fixed $z$, and $c$ is in a compact
subset of $C$, then the map $(c,z) \mapsto \phi_c(z)$ is analytic
in both variables.
\end{coro}

\begin{proof}
Use the same estimates as in the previous proof:
the maps $(c,z) \mapsto \phi_n(c,z)$ are analytic and the infinite product
converges absolutely and uniformly on a neighborhood $U$ of $0$.
\end{proof} 

\begin{lemma} The Julia set $J_c$ is a compact nonempty set. \end{lemma}

\begin{proof}
(i) The Julia set is bounded:
the Lemma of Boettcher-Fatou implies that every point $z$
with large enough $|z|$ converges to $\infty$. This means
that a whole neighborhood $U$ of $z$ escapes to $\infty$.
In other words, the family ${\mathcal{F}}=\{f_c^n\}_{n \in N}$ is normal,
because every sequence in ${\mathcal{F}}$ converges to the constant function
taking the value $\infty$. \\

(ii) The Julia set is closed: this follows from the definition,
because the Fatou set $F_c$ is open and the complement of an open set is closed. \\

(iii) Assume the Julia set was empty. The family ${\mathcal{F}}=\{f_c^n\}$ would be
normal on $\overline{C}$. This means that for any sequence $f_n$ in ${\mathcal{F}}$,
there is a subsequence $f_{n_k}$ converging to an analytic function
$f:\overline{C} \to \overline{C}$. Because such a function can have only
finitely many zeros and poles, it must be a rational function $P/Q$,
where $P,Q$ are polynomials. If $f_{n_k} \to f$, there are
eventually the same number of zeros of $f_{n_k}$ and $f$. But the number
of zeros of $f_{n_k}$ (counted with multiplicity) grows monotonically.
This contradiction makes  $J_c=\emptyset$ impossible.
\end{proof}

\begin{coro}The Julia set $J_c$ is contained in the {\bf filled in Julia} set $K_c$,
the union of $J_c$ and the bounded components of the {\bf Fatou set} $F_c$.
\end{coro}

\begin{proof} 
Because $J_c$ is bounded and $f$-invariant, every orbit
starting in $J_c$ is bounded and belongs by definition to the
filled-in Julia set. If a point is in a bounded component of $F_c$,
its forward orbit stays bounded and it belongs to the filled in Julia set.
On the other hand, if a point is not in the Julia set or a bounded component
of $F_c$, then it belongs to an unbounded component of the Fatou set $F_c$.
\end{proof} 

{\bf Definition:} A continuous function $G: \mathbb{C} \mapsto \mathbb{R}$ 
is called the potential theoretical {\bf Green function}
of a compact set $K \subset \mathbb{C}$, if $G$ is {\bf harmonic} outside $K$, 
vanishes on $K$ and has the property that
$G(z)-\log(z)$ is bounded near $z=\infty$. \\

\begin{lemma}
The Green function $G_c$ exists for the filled-in Julia set
$K_c$ of the polynomial $f_c$. The map $(z,c) \mapsto G_c(z)$ 
is continuous.
\end{lemma}

\begin{proof} 
The Boettcher-Fatou lemma assures the existence of the 
function $\phi_c$ conjugating $f_c$ with $z \mapsto z^2$ in a 
neighborhood $U_c$ of $\infty$. Define for $z \in U_c$
$$ G_c(z)=\log|\phi_c(z)|   \; . $$
This function is harmonic in $U_c$ and grows like $\log|z|$ because
by Boettcher it satisfies
$|f_c^n(z)| \geq C |z|^{2^{n}}$ for some constant $C$ and so
$$ G_c(z)= \lim_{n \to \infty} 
                 \frac{1}{2^n} \log|f_c^n(z)| \; .  $$
Although $G_c$ is only defined in $U_c$,
there is one and only one extension to all of $C$ which is continuous
and satisfies
\begin{equation}
\label{c}
    G_c(z) = G_c(f_c(z))/2 \; .
\end{equation}
In fact, we define $G_c(z)=0$ for $z \in K_c$, and
$G_c(z)=G(f_c^n(z))/2^n$ otherwise, where $n$ is large enough so
that $f_c^n(z) \in U$.
We know from this extension that $G_c$ is a smooth real analytic function
outside $K_c$. From the {\bf maximum principle}, we know that $G_c(z)>0$ for
$z \in C \setminus K_c$. We have still to show that $G_c$ is continuous
in order to see that it is the Green function. The continuity follows from the
stronger statement: \\

$(z,c) \mapsto G_c(z)$ is jointly continuous. \\
$G_c^{-1}( [0,\epsilon) )$ is open in $C^2$
for all $\epsilon>0$ if and only if there exists $n$ such that
$$  A_n:= \{ (c,z) \; | \; G_c(f_c^n(z)) \geq  2^n \epsilon \}  $$
is closed $\forall \epsilon>0$.
Given $r>0$. There exists a ball of radius $b$ which contains all the sets
$K_c$ for $|c|\leq r$.
For $R \geq G_r(b)$, all the solutions $\xi$ of $G_c(\xi) \geq R$
satisfy $|\xi| \geq b$ if $|c| \leq r$.
The set $B=\{(c,\xi) \; | \; G_c(\xi \geq R\} \cap \{|c| \leq r\}$
is closed. For $n$ large enough, also $A_n \cap \{|c| \leq r\}$ is closed
and $A_n$ is closed.
\end{proof}

\begin{thm}[Douady-Hubbard]
The Mandelbrot set M is connected.
\end{thm}

\begin{proof}
The B\"ottcher function $\phi_c(z)$ can be extended to
$$ S_c:=\{ z \; | \; G_c(z) > G_c(0) \}  \; .  $$
Continue defining
$\phi_c(z):=\sqrt{\phi_c(z^2+c)}$ to get $\phi_c$ having defined in
larger and larger regions. This can be done as long as the
region $\phi_c^{-1}(\{r\})$ is connected (this assures that the derivative
of $\phi_c$ is not vanishing). Because Equation~(1)
gives $G_c(c)=2G_c(0)>G_c(0)$, every $c$ is contained in the set
$S_c$ and the map
$$ \Phi: c \mapsto G_c(c) $$
is well defined. It is analytic outside $M$ and can be written as
$$ \Phi(z)= \lim_{n \to \infty} [f_c^n(c)]^{1/2^n} \; . $$

{\bf Claim:}
$$ \Phi: \overline{C} \setminus M \to \overline{C} \setminus \overline{D} $$
is an analytic diffeomorphism, where $\overline{\mathbb{C}}=\mathbb{C} \cup \{\infty\}$ is the
Riemann sphere. (This implies that the complement of $M$
is simply connected in $\overline{\mathbb{C}}$, which is equivalent to the fact that
$M$ is connected). The picture shows some level curves of the 
function $\phi_k(c)$. The function $(\phi_6(z)$ is
already close to the map $\Phi(z)$ in the sense that the level sets give a hint
about the shape of the Mandelbrot set.  \\

(1) $\Phi$ is analytic outside $M$. This follows from the Corollary.  \\

(2) For $c_n \to M$, we have $|\Phi(c_n)| \to 1$.
Proof. This follows from the continuity of the Green function. \\\

(3) The map $\Phi$ is {\bf proper} \index{proper}, meaning that the inverse of any
compact set is compact. \\ Given a compact set $K \subset C \setminus D$. The two compact sets
$D$ and $K$ have positive distance. Assume $\phi^{-1}(K)$ is
not compact. Then, there exists a sequence $c_n \in \Phi^{-1}(K)$ with
$c_n \to c_0 \in M$ so that $|\Phi(c_n)| \to 1$. This
is not possible because $\Phi(c_n) \in K$ is bounded away from $D$. \\

(4) The map $\Phi$ is open (meaning that it maps open sets into open sets). This
follows from the fact that $\Phi$ is analytic. (This fact is called
{\bf open mapping theorem} (see \cite{Conway1978} p. 95)) \\

(5) The map $\Phi$ maps closed sets into closed sets. \\
A proper, continuous map $\Phi: X \to Y$ between two
locally compact metric spaces $X,Y$ has this property.
Proof. Given a closed set $A \subset X$. Take a sequence $\Phi(a_n)$
in $\Phi(A)$ which converges to $b \in Y$. Take a compact neighborhood
$K$ of $b$ (use local compactness of $Y$).
Then $\Phi^{-1}(K \cap \overline{\phi(A)})$ is compact and contains
almost all $a_n$. The sequence $a_n$ contains therefore an accumulation
point $a \in X$. The continuity implies $\Phi(a_n) \to \Phi(a)=b$
for a subsequence so that $b \in \Phi(K)$. Consequently $\Phi(K)$ is closed. \\

(6) $\Phi$ is surjective. \\
The image of $\Phi(\overline{C} \setminus M)$ is an open subset of set
$\overline{C} \setminus \overline{D}$ because $\Phi$ is open. The image of
the boundary of $M$ is (use (5))
a closed subset of $\overline{C} \setminus D$
which coincides with the boundary of $D$ because the boxed statement about
the Green function showed
$G_c(c) \to 0$ as $c \to M$. \\

(7) $\Phi$ is injective. \\
Because the map $\Phi$ is proper, the inverse image $\phi^{-1}(s)$ of a point $s$
is finite. There exists therefore a curve $\Gamma$ enclosing all
points of $\Phi^{-1}(s)$. Let $\sharp A$ denote the number of elements in $A$.
By the {\bf argument principle} (see \cite{AhlforsComplexAnalysis} p. 152),
we have
$$ \sharp (\phi^{-1}(s)) = \frac{1}{2 \pi i} 
  \int_{\Gamma} \frac{\Phi'(z)}{\Phi(z)-s} \; dz  \;  $$
and this number is locally constant. Given $M>0$,
we can find a curve $\Gamma$ which
works simultaneously for all $|s| \leq M$.
Because $\Phi$ is surjective and
$\sharp (\phi^{-1}(\infty))=1$, we get that $\sharp (\phi^{-1}(s))=1$ for all
$z \in C \setminus \overline{D}$ and $\phi$ is injective. \\

(8) The map $\Phi^{-1}$ exists on $C \setminus D$ and is analytic. \\
Because an injective, differentiable and open map has a differentiable 
inverse, (this is called {\bf Goursat's theorem}), the inverse is analytic.
see \cite{Conway1978} p. 96). 
\end{proof}

\paragraph{}
We end with some history of the classical Mandelbrot set.
In 1879, {\bf Arthur Cayley} posed the problem to study the 
regions $F_p$ in the plane, where the {\bf Newton iteration} of a polynomial 
like $g(z) = z^3-1$ converges to some root $p$ of $f$. 
The Newton iteration of a function $g$ is $T(z)=z-g(z)/g'(z)$ is by nature not
a polynomial map but a {\bf rational map}. The study of rational maps has become
a subject of interest by its own \cite{Beardon}. 

\paragraph{}
{\bf Gaston Julia} (1893-1978) and {\bf Pierre Fatou}
(1879-1929) worked on the iteration of
analytic maps. Julia and Fatou sets are called after them. Julia and Fatou 
both competed for the 1918 'grand prix' of the French Academie of Sciences and 
produced similar results. This produced a priority dispute. Julia lost his 
nose during WWI and had to wear a leather strap across his face.
He continued with his research while recovering in the hospital. 

\paragraph{}
{\bf Robert Brooks and Peter Matelski} produced in 1978 
the first picture of the Mandelbrot set in the context of 
Kleinian groups \cite{MilnorNotes}. Their paper had the title
``The dynamics of $2$-generator subgroups of ${\rm PSL}(2,\,C)$". The
defined $\tilde{M} = \{ c \: | \; f_c$ has a stable periodic orbit $\; \}$. 
This set is now called {\bf Brooks-Matelski set} and is believed
to be the interior of the Mandelbrot set $M$. If the later were locally 
connected, this would be true: ${\rm int}(M) = \tilde{M}$. 

\paragraph{}
{\bf John Hubbard} made better pictures of a quite different
parameter space arising from Newton's method for cubics. 
Hubbard was inspired by a question from a calculus student. 
{\bf Benoit Mandelbrot}, perhaps inspired by Hubbard, made corresponding 
pictures in 1980 for quadratic polynomials. He conjectured the set $M$ 
is disconnected because his computer pictures showed ``dust" with no connections
to the main body of $M$. It is amusing that the journal's editorial staff 
removed that dust, assuming it was a problem of the printer.

\paragraph{}
{\bf John Milnor} writes in his book of 1991 \cite{MilnorNotes}:
{\it "Although Mandelbrot's statements in this first paper
were not completely right, he deserves a great
deal of credit for being the first to point out the extremely complicated
geometry associated with the parameter space for quadratic maps.
His major achievement has been to demonstrate to a very wide audience 
that such complicated fractal objects play an important role in a 
number of mathematical sciences."}

\paragraph{}
{\bf Adrien Douady and John Hubbard} prove the connectivity of $M$ in 1982 (see \cite{DouadyHubbard}).
This was a mathematical breakthrough. In that paper, the name {\bf ``Mandelbrot set"}
was introduced. The paper provided a firm foundation for its mathematical study. 
We followed on this handout their proof. Note that the Mandelbrot set is also 
{\bf simply connected}, but this is easier to show. Both statements use that a subset of
the plane is connected if and only if the complement is simply connected. 

\paragraph{}
One of the first things which comes to mind, when talking
about fractals, is the Mandelbrot set. But it is not a "fractal": in 1998,
{\bf Mitsuhiro Shishikura} has shown that its Hausdorff dimension of 
$M$ is $2$. \cite{Shishikura1991}. 
Also for higher dimensional polynomials, one can define Julia and 
Mandelbrot sets. For cubic polynomials $f_{a,b}(z) = z^3-3a^2z+b$, 
define the {\bf cubic locus set}
$\{ (a,b) \in C^2 \; | \;  K_{a,b}$ is connected  $\; \}$,
where $K_{a,b}$ is the {\bf prisoner set} 
$K_{a,b} = \{ z \; | \; f_{a,b}^n(z)$ stays bounded. $\}$.  
{\bf Bodil Branner} showed around 1985, that the cubic locus set is connected.
See \cite{BrannerHubbard1988}.
This generalizes the Douady-Hubbard theorem. 

\paragraph{}
A major open problem about polynomial complex dynamical systems is whether the 
Mandelbrot set is locally connected or not. A subset $M$ of the plane is
called {\bf locally connected}, if at every point $x \in M$ if every 
neighborhood of $x$ contains a neighborhood, in which $M$ is connected. 
A locally connected set does not need to be connected (two disjoint disks 
in the plane are locally connected but not connected). A connected set 
does not need to be locally connected. An example is the union of the 
graph of $\sin(1/x)$ and the $y$-axes. 

\begin{figure}[!htpb]

\scalebox{0.28}{\includegraphics{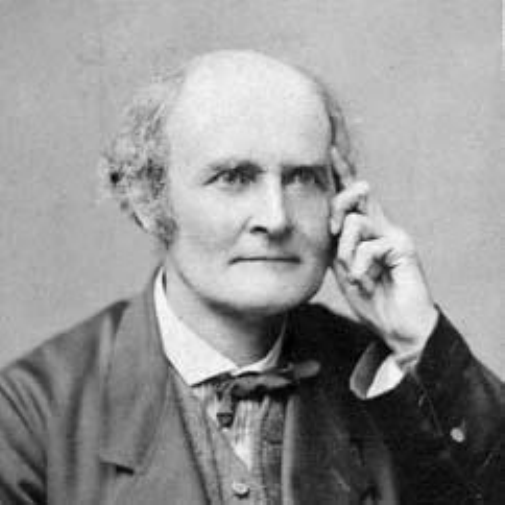}}
\scalebox{0.30}{\includegraphics{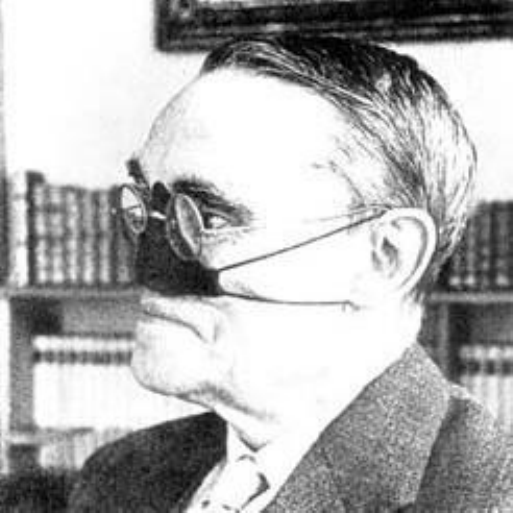}}
\scalebox{0.30}{\includegraphics{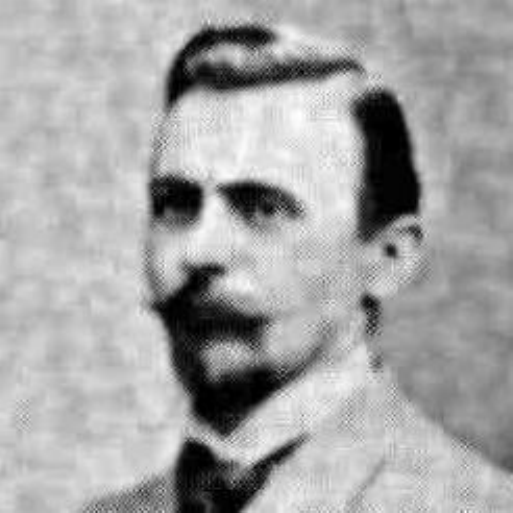}}
\scalebox{0.08}{\includegraphics{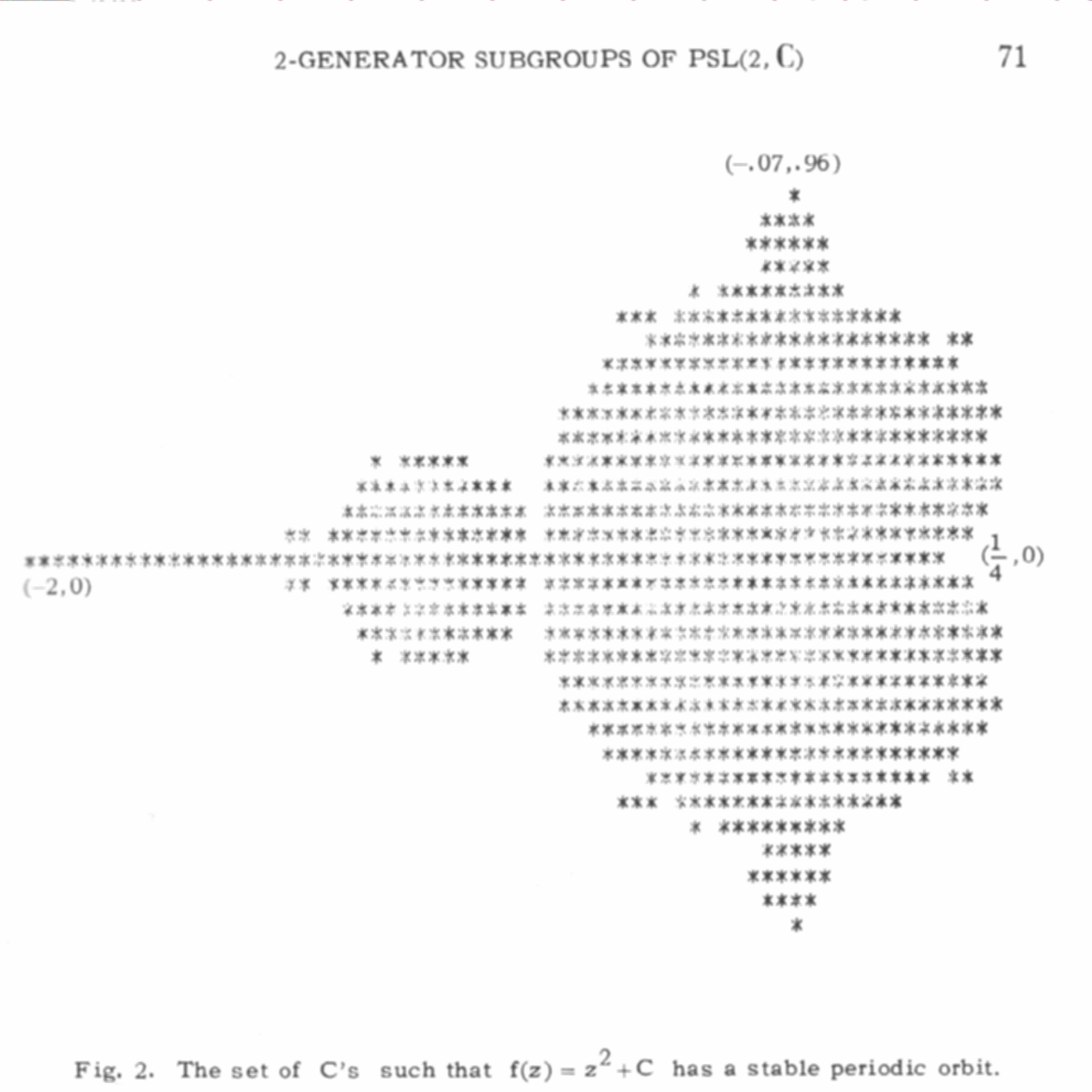}}
\scalebox{0.40}{\includegraphics{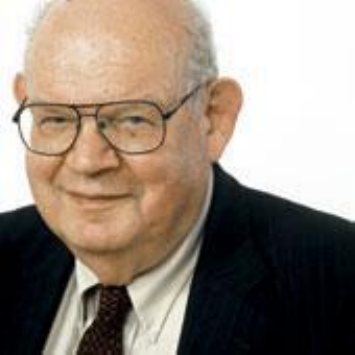}}
\\
\scalebox{0.18}{\includegraphics{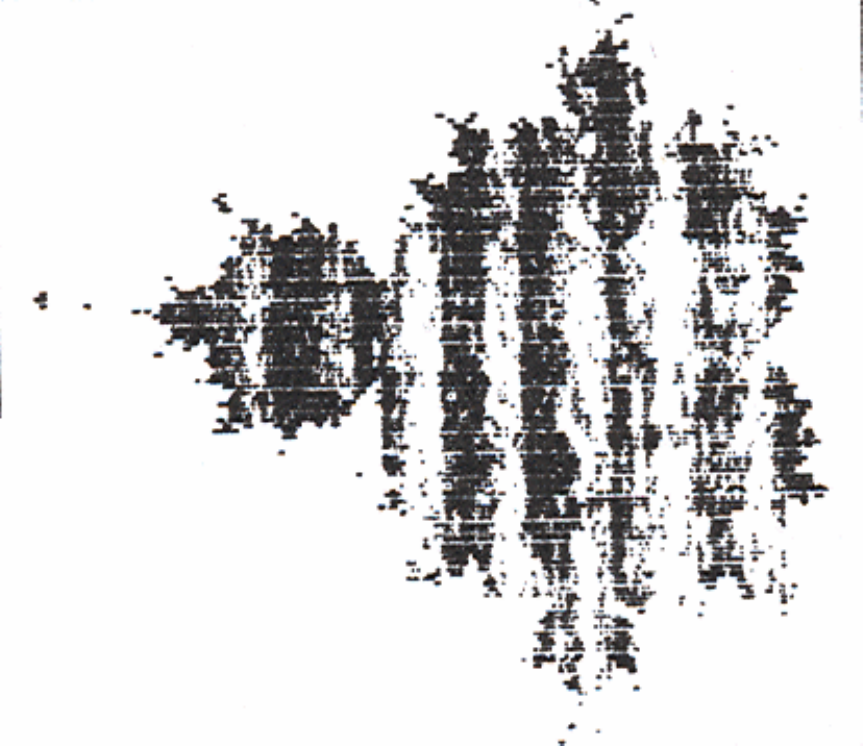}}
\scalebox{0.40}{\includegraphics{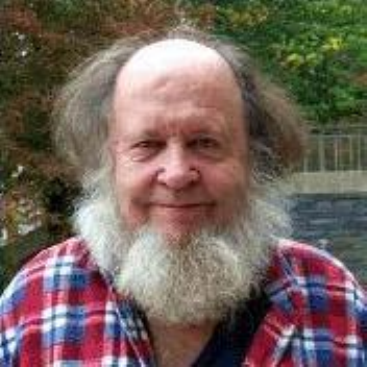}}
\scalebox{0.40}{\includegraphics{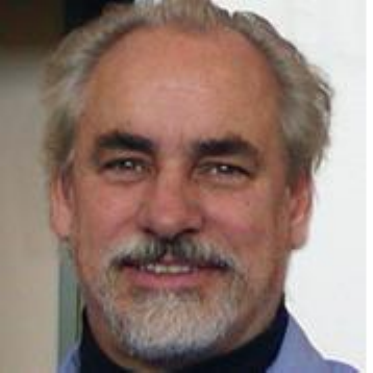}}
\scalebox{0.24}{\includegraphics{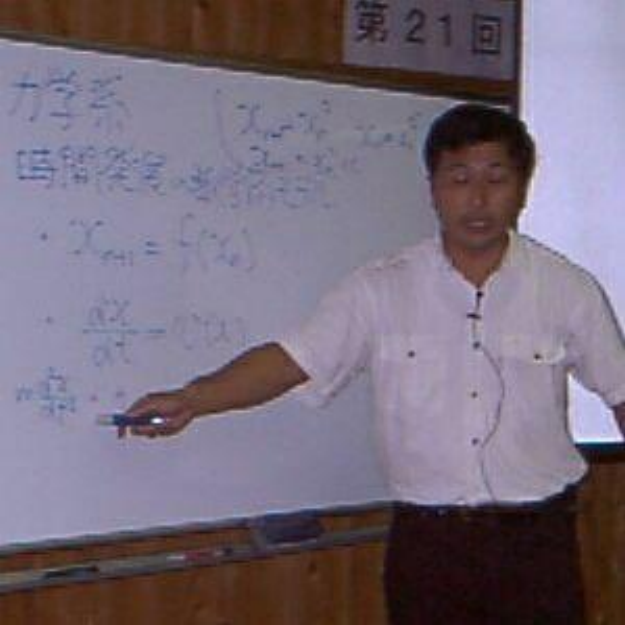}}
\scalebox{0.40}{\includegraphics{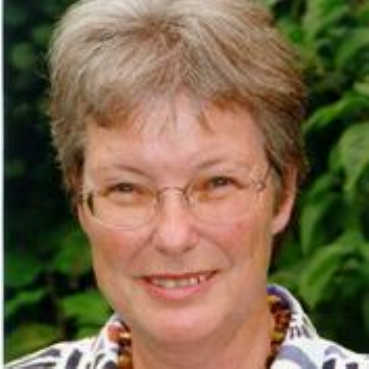}}
\caption{
Cayley, Julia, Fatou, Books,Matelski, Mandelbrot, Douady,
Hubbard, Shishikura, Branner.
}
\end{figure}

\bibliographystyle{plain}

\end{document}